\documentclass[a4paper,10pt, leqno]{amsart}
\date{September 10, 2021}
\usepackage{amsmath, amssymb, graphics, setspace}

\usepackage{graphicx}

\title[Symmetries of cross caps]{Symmetries of cross caps}
\author[A. Honda]{Atsufumi Honda}
\address[Atsufumi Honda]{
Department of Applied Mathematics, 
Faculty of Engineering, Yokohama National University,
79-5 Tokiwadai, Hodogaya, Yokohama 240-8501, Japan
}
\email{honda-atsufumi-kp@ynu.ac.jp}

\author[K. Naokawa]{Kosuke Naokawa}
\address[Kosuke Naokawa]{%
Department of Computer Science, 
Faculty of Applied Information Science,
Hiroshima Institute of Technology,  
2-1-1 Miyake, Saeki, Hiroshima 731-5193, Japan
}
\email{k.naokawa.ec@cc.it-hiroshima.ac.jp}

\author[K. Saji]{Kentaro Saji}
\address[Kentaro Saji]{%
  Department of Mathematics,
  Faculty of Science,
  Kobe University,
  Rokko, Kobe 657-8501}
\email{saji@math.kobe-u.ac.jp}
\author[M. Umehara]{Masaaki Umehara}
\address[Masaaki Umehara]{%
  Department of Mathematical and Computing Sciences,
  Tokyo Institute of Technology,
  Tokyo 152-8552, Japan}
\email{umehara@is.titech.ac.jp}
\author[K. Yamada]{Kotaro Yamada}
\address[Kotaro Yamada]{%
  Department of Mathematics,
  Tokyo Institute of Technology,
  Tokyo 152-8551, Japan}
\email{kotaro@math.titech.ac.jp}


\thanks{%
The first author was partially supported by 
Grant-in-Aid for Early-Career Scientists
 No.~19K14526 and No. 20H01801. 
The second author was partially supported by 
Grant-in-Aid for Young Scientists (B) No.~17K14197,
and the third author
was 
partially supported by  Grant-in-Aid for 
Scientific Research (C) No.\ 18K03301.
The forth author 
was partially 
supported by Grant-in-Aid for 
Scientific Research (B) No.\ 21H00981.
The fifth author 
was partially 
supported by Grant-in-Aid for 
Scientific Research (B) No.\ 17H02839.
}%

\newcommand{\vect}[1]{{\boldsymbol{#1}}}
\newcommand{\R}{\boldsymbol{R}}
\newcommand{\mb}[1]{{\mathbf #1}}
\newcommand{\pmt}[1]{{\begin{pmatrix} #1  \end{pmatrix}}}
\renewcommand{\phi}{\varphi}
\renewcommand{\epsilon}{\varepsilon}
\newcommand{\dy}{\displaystyle}
\numberwithin{equation}{section}
\newtheorem{Theorem}{Theorem}[section]

\newtheorem{Corollary}[Theorem]{Corollary}
\newtheorem{Lemma}[Theorem]{Lemma}
\newtheorem{Fact}[Theorem]{Fact}
\theoremstyle{definition}
\newtheorem{Def}[Theorem]{Definition}
\newtheorem{Remark}[Theorem]{Remark}
\newtheorem{Example}[Theorem]{Example}
\newtheorem*{acknowledgments}{Acknowledgments}

 \makeatletter
       \def\@makefnmark{%
               \leavevmode
               \raise.9ex\hbox{\check@mathfonts
                       \fontsize\sf@size\z@\normalfont%
                               \@thefnmark}%
       }
       \makeatother

\usepackage[active]{srcltx}
\begin{document}
\maketitle

\begin{abstract}
It is well-known that cross caps on surfaces
in the Euclidean 3-space can be expressed in Bruce-West's
normal form, 
which is a special local coordinate system centered
at the singular point.
In this paper, we show a certain kind of uniqueness of
such a coordinate system. In particular, the
functions associated with this coordinate system
produce new invariants on cross cap singular points.
Using them, we characterize the possible symmetries 
on cross caps.
\end{abstract}

\section*{Introduction}
Cross caps (which are also called Whitney's umbrellas, see 
Figure \ref{fig:cr-explain}, left)
are the only singular points of stable maps of surfaces to
3-manifolds, and investigated by
several geometers  \cite{BT,BW,I6, FH, I1,I2,I3,HHNSUY-2015,
HHNUY-2014, HNUY, I4,I5}.

Let $f\colon{}U\to\R^3$ be a $C^{\infty}$-map  
from a domain $U$ in the Euclidean plane $\R^2$
into the Euclidean space $\R^3$ 
having a cross cap at $p\in U$.
Then there exists a 
local coordinate system $(u,v)$ centered at 
$p$ satisfying $f_v(0,0)=\vect{0}$,
where $f_v:=\partial f/\partial v$ and $\vect{0}:=(0,0,0)$.
Since the rank of the Jacobi matrix of $f$ is one,
we have $f_{u}(0,0)\neq \vect{0}$.
By the well-known criterion for a cross cap
by Whitney, $\{f_u,f_{uv},f_{vv}\}$ is 
linearly independent at $(0,0)$.

\begin{figure}[thb]
 \centering
  \includegraphics[height=3.6 cm]{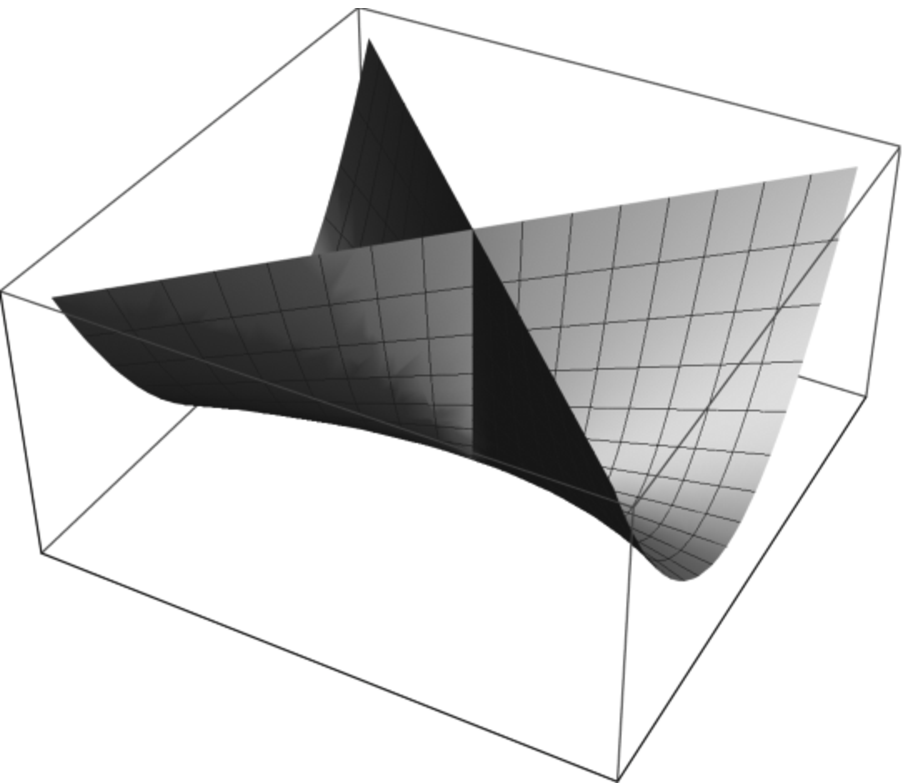}\quad \qquad
  \includegraphics[height=5.1 cm]{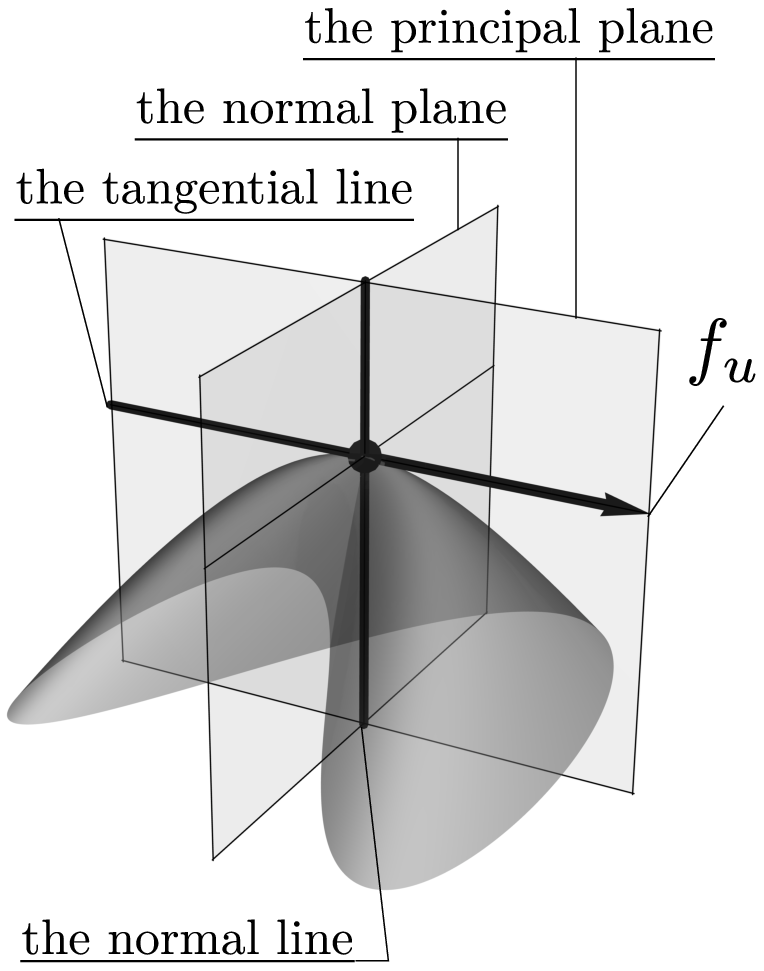}
\caption{The image of the standard cross cap (left) and
the tangent and normal lines, 
and principal and normal planes of 
a cross cap with three symmetries with $a_{20}>0$ (right)}\label{fig:cr-explain}
\end{figure}

\begin{Def}\label{def:cr}
We fix a cross cap $f:(U;u,v)\to\R^3$
satisfying $f_v(0,0)=\mb 0$.
We call the line
$$
\left\{f(0,0) + tf_u(0,0)\,;\,t\in\R\right\}
$$
the {\it tangent line} at the cross cap.
The plane in $\R^3$ passing through $f(0,0)$ spanned
by $f_u(0,0)$ and $f_{vv}(0,0)$ is called the {\it principal plane}
or the {\it co-normal plane}.
On the other hand, the plane
passing through $f(0,0)$ perpendicular to the tangent line 
is called the {\it normal plane}
(see Figure \ref{fig:cr-explain}, right).
Moreover, the line obtained 
as the intersection of the principal plane
and the normal plane is called the {\it normal line}. 
\end{Def}

We note that the tangent line, the principal plane
and the normal plane do not depend on 
the choice of admissible coordinate systems 
(cf. \cite{HHNUY-2014} and \cite{HHNSUY-2015}).

A local coordinate system $(u,v)=(u(r_1,r_2),v(r_1,r_2))$
of $(\R^2;r_1,r_2)$ is called {\it positive} 
(resp. {\it negative}) if the Jacobian
$u_{r_1}v_{r_2}-v_{r_1}u_{r_2}$
is positive (resp. negative).
The following fact is known:

\begin{Fact}[cf. \cite{BW} and \cite{west}, see also
 \cite{FH} and \cite{HNUY}]\label{thm:west-2}
Let $f\colon{}U\to\R^3$ be a $C^{\infty}$-map having a cross cap at $p\in U$.
Then there exist a positive local coordinate system 
$(u,v)$
centered at $p$ and
$C^\infty$-functions $a(u,v),\,\,b(v)$ satisfying
\begin{equation}\label{eq:cb}
b(0)=b'(0)=b''(0)=0,\quad
a(0,0)=a_u(0,0)=a_v(0,0)=0,\quad a_{vv}(0,0)>0
\end{equation}
such that
 \begin{equation}\label{eq:west-2}
   f(u,v)=
        \left(
	   u,
	   uv +b(v),
	   a(u,v)
	\right)
\end{equation}
after composing an appropriate isometric motion in $\R^3$.
In this expression, $\vect{e}_1:=(1,0,0)$ points
in the direction of the tangent line, 
$\vect{e}_3:=(0,0,1)$
and $\vect{e}_1$ 
span the principal plane, and
$\vect{e}_2:=(0,1,0)$ and $\vect{e}_3$ 
span the normal plane.
\end{Fact}

The coordinate system $(u,v)$ in Fact \ref{thm:west-2}
is called a {\it Bruce-West coordinate system},
and the equation \eqref{eq:west-2}
is called the {\it normal form} of a cross cap.

The geometric invariance of the coefficients of 
Maclaurin series of $a(u,v)$ and $b(v)$
has
been pointed out and discussed in \cite{FH, HHNUY-2014, HNUY}.
However, these functions are determined up to 
additions of flat functions in general, 
and there are no references which assert that the 
functions $a(u,v)$ and $b(v)$ themselves 
are geometric invariants as far as the authors know.
The purpose of this paper is
to prove the following local rigidity theorem of 
Bruce-West's coordinates, which removes 
these ambiguities in $a(u,v)$ and $b(v)$ by flat functions:

\medskip
\noindent
{\bf Theorem A} (Rigidity of Bruce-West's coordinates).
\label{thm:W}
{\it Let $(U_i;u_i,v_i)$ $(i=1,2)$  be 
positive coordinate neighborhoods centered at
 $p_i$ in $\R^2$, and 
$f_i:U_i\to \R^3$ 
$C^\infty$-maps satisfying $f_i(0,0)=\mb 0$
each of which has a cross cap singular point at $p_i\in U_i$.
Suppose that 
$f_i$ $(i=1,2)$ are written in the following normal forms
  \begin{equation}\label{eq:west-2b}
   f_i(u_i,v_i)=
        \left(
	   u_i,
	   u_iv_i +b_i(v_i),
	   a_i(u_i,v_i)
	\right)
\end{equation}
and $f_1(U_1)\subset f_2(U_2)$.
Then there exist a pair of
neighborhoods $V_i(\subset U_i)$ of $p_i$ $(i=1,2)$
and an orientation preserving diffeomorphism $\phi:V_1\to V_2$ such that 
$$ 
f_1=f_2\circ \phi,\quad u_1=u_2\circ \phi(u_1,v_1),\quad
v_1=v_2\circ \phi(u_1,v_1).
$$ 
As a consequence, there exists a positive number $\epsilon$ such that
$$ 
a_1(u,v)=a_2(u,v),
\quad
b_1(v)=b_2(v)
$$
hold for $u,v$ satisfying $|u|, |v|<\epsilon$.}

\medskip
Regarding the above results, we give the following definition:

\begin{Def}
We call $a(u,v)$ the {\it first characteristic function}
and $b(v)$ the {\it second characteristic function}.
\end{Def}

We also give the following definition:

\begin{Def}\label{def:map-c}
Let $U_i$ $(i=1,2)$  be a neighborhood of $p_i$ in $\R^2$.
For each $i=1,2$, let
$f_i:U_i\to \R^3$ 
be a $C^\infty$-map having a cross cap singular point
at $p_i\in U_i$.
We say that $f_2$ is {\it congruent} to $f_1$
as a map germ
if there exist an isometry $T$ of $\R^3$
and a diffeomophism germ $\phi$ satisfying $\phi(p_1)=p_2$
such that $f_1=T\circ f_2 \circ \phi$
holds on a sufficiently small neighborhood of $p_1$.
Moreover, if $T$ and $\phi$ are both orientation preserving,
then $f_2$ is said to be {\it positively congruent} to $f_1$.
\end{Def}

As an application of Theorem A, we obtain the following:

\medskip
\noindent
{\bf Theorem B.}
{\it The two characteristic functions $a(u,v)$ and $b(v)$ 
can be considered as geometric invariants
of positive congruence classes on cross cap germs.
}

\medskip
Moreover, as an application,
we show the following three corollaries:

\medskip
\noindent
{\bf Corollary C-1.}\label{Cor:d-1}
{\it
Let $f:U\to \R^3$ be a $C^\infty$-map defined on
an open subset $U$ of $\R^2$ and
$p\in U$ a cross cap singular point.
Then the following three assertions 
are equivalent:
\begin{enumerate}
\item[(i${}_1$)] The image germ of $f$ at $p$ is
invariant under the reflection $T_1$
with respect to the principal plane.
\item[(ii${}_1$)]
There exist a connected open
neighborhood $W(\subset U)$ of $p$
and an involutive orientation reversing
$C^\infty$-diffeomorphism $\phi:W\to W$
such that $f\circ \phi=T_1\circ f$ on $W$.
\item[(iii${}_1$)]
The characteristic functions satisfy 
$a(u,-v)=a(u,v)$ and $b(v)=-b(-v)$. 
\end{enumerate}
}

\medskip
\noindent
{\bf Corollary C-2.}\label{Cor:d-2}
{\it Let $f:U\to \R^3$ be a $C^\infty$-map defined on
an open subset $U$ of $\R^2$ and
$p\in U$ a cross cap singular point.
Then the following three assertions 
are equivalent:
\begin{enumerate}
\item[(i${}_2$)] The image germ of $f$ at $p$ is
invariant under the 
reflection $T_2$ with respect to the normal plane.
\item[(ii${}_2$)]
There exist a connected open
neighborhood $W(\subset U)$ of $p$
and an involutive orientation preserving
$C^\infty$-diffeomorphism $\phi:W\to W$
such that $f\circ \phi=T_2\circ f$ on $W$.
\item[(iii${}_2$)]
The characteristic functions satisfy 
$a(u,v)=a(-u,-v)$ and $b(v)=b(-v)$.
\end{enumerate}
}

\medskip
\noindent
{\bf Corollary C-3.}\label{Cor:d-3}
{\it 
Let $f:U\to \R^3$ be a $C^\infty$-map defined on
an open subset $U$ of $\R^2$ and
$p\in U$ a cross cap singular point.
Then the following three assertions 
are equivalent:
\begin{enumerate}
\item[(i${}_3$)] The image germ of $f$ at $p$ is
invariant under the 
$180^\circ$-rotation $T_3$ with respect to the
normal line.
\item[(ii${}_3$)]
There exist a connected open
neighborhood $W(\subset U)$ of $p$
and an involutive 
orientation reversing 
$C^\infty$-diffeomorphism $\phi:W\to W$
such that $f\circ \phi=T_3\circ f$ on $W$.
\item[(iii${}_3$)]
The characteristic functions satisfy 
$a(-u,v)=a(u,v)$
and $b(v)=0$.
\end{enumerate}
}

\rm
\medskip
Finally, we remark that Corollary C-$j$ ($j=1,2,3$) 
can be generalized for other space forms:
We denote by $M^3(c)$ the simply connected 
complete Riemannian $3$-manifold of
constant sectional curvature $c$.
For a germ of cross cap in $M^3(c)$ ($c\ne 0$),
we can define the notions of
{\it tangent line}, {\it normal line}
{\it principal plane} and
{\it normal plane},
like as in the case of the Euclidean 3-space.
We then denote by $T_1$ (resp. $T_2$)
the reflection with respect to the principal plane (resp. normal plane).
On the other hand, let $T_3$ be the
$180^\circ$-rotation with respect to the
normal line. The following assertion holds:

\medskip
{\bf Proposition D.}
{\it Let $f:U\to M^3(c)$ be a $C^\infty$-map defined on
an open subset $U$ of $\R^2$ and
$p\in U$ a cross cap singular point.
Then 
the following three assertions
are equivalent for each $j=1,2,3$:
\begin{enumerate}
\item[(i)] The image germ of $f$ at $p$ is
invariant under the 
isometry $T_j$ of $M^3(c)$.
\item[(ii)]
There exist a connected open
neighborhood $W(\subset U)$ of $p$
and a $C^\infty$-involution 
$\phi:W\to W$
such that $f\circ \phi=T_j\circ f$ on $W$.
\item[(iii)]
The two characteristic functions $a(u,v)$ and $b(v)$
$($cf. \eqref{eq:forD}$)$
satisfy 
{\rm (iii${}_j$)} of Corollary C-${j}$.
\end{enumerate}
}

\medskip
\section{Proofs}

We first prepare a definition and a lemma:

\begin{Def}\label{def:admissible}
Let $f:U\to \R^3$ be a $C^\infty$-map with a
cross cap singular point at $p\in U$
having a unit normal vector field $\nu$ of $f$
defined on $U\setminus \{p\}$.
Then an open neighborhood $V(\subset U)$ of $p$ 
is said to be {\it admissible} if 
\begin{enumerate}
\item[(a1)] the closure $\overline{V}$ of $f$ is compact and contained in $U$,
\item[(a2)] $(f|_V)^{-1}(f(p))=\{p\}$,
where $f|_V$ is the restriction of $f$ to the subset $V$, 
\item[(a3)] 
the map $L:=(f,\nu):V\setminus \{p\}\to \R^3\times
S^2$ is an injective immersion, and
\item[(a4)]
the set 
\begin{equation}\label{eq:Vstar0}
A:=\Big\{q\in V \setminus \{p\} \,;\, 
\exists q'\in V\setminus \{p,q\} \text{ such that }
 f(q)=f(q') \Big\}\cup \{p\}
\end{equation}
is the image of a  regular curve in $V$  
passing through $p$
such that $V\setminus A$ consists of two
connected components.
\end{enumerate}
\end{Def}

\begin{Lemma}\label{lem:l2}
Let 
$f:U\to \R^3$ be a $C^\infty$-map with a
cross cap singular point at $p\in U$.
Then there exists an
admissible neighborhood $V(\subset U)$ 
of $p$.
\end{Lemma}

\begin{proof}
The standard cross cap  $f_0:\R^2\to \R^3$ is defined by
$
f_0(u,v):=(u,uv,v^2)
$
(see Figure \ref{fig:cr-explain}, left).
Since $f$ as a map germ at $p$
is right-left equivalent to the standard cross cap,
there exist
\begin{itemize}
\item an open neighborhood $U_0(\subset \R^2)$ of the origin $(0,0)$,
\item a $C^\infty$-map $\phi:U_0\to \R^2$ 
giving a
diffeomorphism between $U_0$ and $\phi(U_0)(\subset U)$,
\item an open neighborhood $\Omega(\subset \R^3)$ of $f(p)$ and
\item a $C^\infty$-map $\Phi:(\Omega,f(p))\to (\R^3,\mb 0)$
giving a diffeomorphism between $\Omega$ and $\Phi(\Omega)$ 
\end{itemize}
such that
$\Phi\circ f\circ \phi=f_0$ holds on $U_0$.
Then the closure of the open disk $D_r$ of radius $r$ centered at the origin
is contained in $U_0$ for sufficiently small $r(>0)$.
We fix such an $r$ and set $V_0:=D_r$.
Since
$f_0^{-1}(f_0(0,0))=\{(0,0)\}$, the map $f$ satisfies (a1) and (a2).
On the other hand, we consider the regular curve parametrizing the $v$-axis
$
\gamma_0(v):=\{(0,v)\in V_0\,;\, v\in \R\}
$,
whose image is the closure of the self-intersection set of $f_0$ in $V_0$,
because $f_0(0,v)=f_0(0,-v)$.
Moreover, 
\begin{equation}\label{eq:n0}
\nu_0(u,v):=
\frac{1}{\sqrt{u^2+4 v^2+4 v^4}}
\left(2 v^2,-2 v,u\right) \qquad ((u,v)\in \R^2\setminus \{(0,0)\})
\end{equation}
is a unit normal vector field of $f_0$ on $\R^2\setminus \{(0,0)\}$.
By computing $\nu_0$ at each point of $A_0\setminus \{(0,0)\}$, 
it can be easily observed that, 
at each point $q_0\in A_0\setminus \{(0,0)\}$, two sheets of $f_0$ at $f_0(q_0)$
meet transversally.
We set $V:=\phi(V_0)$.
Since $\Phi\circ f\circ \phi=f_0$,
the set (as the image of the regular curve $\Phi\circ \gamma_0$)
$
A:=\phi(A_0)
$
coincides with the set given in
\eqref{eq:Vstar0}. So (A) satisfies (a4).
Moreover, 
at each point $q\in A\setminus \{p\}$, two sheets of $f$ at $f(q)$
meet transversally, as well as $f_0$. 
So the map $L$ satisfies (a3).
\end{proof}

\begin{Corollary}\label{cor:VU}
Let $f$ be as above,
and let $V(\subset U)$ be an admissible neighborhood 
of $p$.
If
$D_{\epsilon}(p)$ is the open disk of radius $\epsilon(>0)$
centered at $p$ satisfying
$D_{\epsilon}(p)\subset V$,
then $L=(f,\nu)$ gives a homeomorphism
between $B:=\overline{V}\setminus D_{\epsilon}(p)$ 
and $L(B)$.
Moreover, by setting $O:=V\setminus \overline{D_{\epsilon}(p)}$,
the restriction $L|_{O}$ of $L$ to the open subset $O$
is an embedding.
\end{Corollary}

\begin{proof}
Since $B$ is compact,
$
L|_{B}:
B
\to L(B)
$
is a homeomorphism, because of the fact that 
a continuous map from a compact space to a
Hausdorff space is a closed map.
Since $O$ is a subset of $\overline{W} \setminus D_{\epsilon}$,
the map $L|_{O}$ gives a homeomorphism between $O$ and $L(O)$.
Since $L$ is an injective immersion on $V$,
we can conclude that $L|_{O}$ is an embedding.
\end{proof}

We let $(U_i;u_i,v_i)$ $(i=1,2)$  be a coordinate neighborhood centered at
$p_i$ in $\R^2$ and 
$f_i:U_i\to \R^3$ a $C^\infty$-map satisfying $f_i(0,0)=\mb 0$
which has a cross cap singular point at $p_i\in U_i$.
Suppose that 
$f_i$ $(i=1,2)$ are written in the normal forms
and $f_1(U_1)$ is a subset of $f_2(U_2)$.
By Lemma \ref{lem:l2}, for each $i=1,2$,
we can take an admissible neighborhood $V_i(\subset U_i)$ of $p_i$.
Since the canonical form $f_0:\R^2\to \R^3$ is a proper map,
by \cite[Corollary 1.15]{HNSUY}, we may assume that
$(f_i|_{V_i})^{-1}(f_i(p_i))=\{p_i\}$  and $f_i$ is
$V_i$-proper in the sense of \cite[Definition 1.1]{HNSUY}
for each $i=1,2$.
Then, by \cite[Theorem 1.16]{HNSUY},
the condition $f_1(U_1)\subset f_2(U_2)$ implies that
there exist relatively compact neighborhoods $W_i(\subset V_i)$ ($i=1,2$)
such that $\overline{W_i}\subset V_i$ and
and $f_1(W_1)\subset f_2(W_2)$.
Again, applying Lemma \ref{lem:l2},
we take an admissible neighborhood $W'_1$ of $p_1$
such that $\overline{W'_1}\subset W_1$.
Here, we replace $W_1$ by $W'_1$ and
$W_2$ by $V_2$, and then obtain
admissible neighborhoods $W_i$ ($i=1,2$) of $p_i$
satisfying 
$$
\overline{W_1}\subset U_1,\quad
\overline{W_2}\subset U_2,\quad
f_1(W_1)\subset f_2(W_2).
$$
For each $i\in \{1,2\}$, we set
\begin{equation}\label{eq:Vstar}
A_i:=\Big\{q\in W_i\setminus \{p_i\} 
\,;\, \exists q'\in W_i\setminus \{p_i,q\} \text{ such that }
f_i(q)=f_i(q') \Big\}\cup \{p_i\}.
\end{equation}
Since $W_i$ ($i=1,2$)
is admissible,
(a4) in Definition \ref{def:admissible}
 implies that
$W_i\setminus A_i$ consists of two
connected components, denoted by $W'_i$ and $W''_i$.
Since
$f_1(W_1)\subset f_2(W_2)$,
it is obvious that $f_1(A_1)\subset f_2(A_2)$.
So we may assume that
$f_1(W'_1)\subset f_2(W'_2)$
and $f_1(W''_1)\subset f_2(W''_2)$.
We then replace by $\nu_2$ by $-\nu_2$ if necessary,
and may also assume that
$
L_1(W'_1)\subset L_2(W'_2).
$
Since the inward (resp. the outward)
normal vector field $\nu_i$ of $f_i$ on $W'_i$
turns to be an outward (resp. an inward)
normal vector field $\nu_i$ on $W''_i$,
the fact $L_1(W'_1)\subset L_2(W'_2)$ implies
$L_1(W''_1)\subset L_2(W''_2)$.
(For example, for the standard cross cap $f_0$,
$\nu_0(u,0) = (0,0,1)$ (resp. $\nu_0(u,0) = (0,0,-1)$)
given in \eqref{eq:n0}
is the inward (resp. outward) normal vector when $u>0$ (resp. $u<0$), see 
Figure \ref{fig:cr-explain}, left.)
So we have
\begin{equation}\label{eq:L1L2}
L_1(W^*_1)\subset L_2(W^*_2),
\end{equation}
where
$
W^*_i:=W_i\setminus \{p_i\}.
$
Since $L_2$ is injective (cf.~(a3)),
the map $\phi:=L_2^{-1}\circ L_1$ from $W_1^*$ to
$W_2^*$ is defined.
The following assertion holds:

\begin{Lemma}
The set $\phi(W_1^*)$ is open in $W_2^*$,
and $\phi$ gives a diffeomorphm between $W^*_1$ and $\phi(W_1^*)$.
\end{Lemma}

\begin{proof}
We let $\epsilon_i$ $(i=1,2)$ be sufficiently small positive numbers.
We set
$$
O_i:=W_i \setminus \overline{D_{\epsilon_i}(p_i)},\quad 
B_i:=\overline{W_i} \setminus D_{\epsilon_i}(p_i) 
\qquad (i=1,2).
$$
We can choose  
$\epsilon_2>0$ so that
$$
L_1(O_1)\subset L_2(O_2), \qquad L_1(B_1)\subset L_2(B_2).
$$
In fact, suppose the first (resp. the second) equality fails
for arbitrary $\epsilon_2(>0)$.
By regarding \eqref{eq:L1L2},
there exist $q_n\in O_1$ (resp. $q_n\in B_1$)
and $q'_n\in D_{1/n}(p_2)$ (resp. $q'_n\in \overline{D_{1/n}(p_2)}$)
such that
$
L_1(q_n)=L_2(q'_n).
$
Since $B_1$ is compact,
we can take a convergent subsequence
$\{q_{i_n}\}_{n=1}^\infty$ of $\{q_n\}_{n=1}^\infty$ converging
to a point $q_\infty\in B_1$.
Since $\lim_{n\to \infty}q'_n=p_2$, we have that
$$
f_1(q_\infty)=\lim_{n\to \infty}f_1(q_{i_n})=\lim_{n\to \infty}f_2(q'_{n})=\mb 0,
$$
and the fact $f_1^{-1}(f_1(p_1))=\{p_1\}$ implies
$q_\infty=p_1$, which contradicts that $q_\infty\in B_1$.
Thus, $L_1(O_1)\subset L_2(O_2)$ (resp. $L_1(B_1)\subset L_2(B_2)$)
holds for some $\epsilon_2$.

By the first assertion of Corollary \ref{cor:VU}, each
$L_i|_{B_i}$ ($i=1,2$) gives a homeomorphism
between $B_i$ and $L_i(B_i)$, and so,
$\phi|_{B_1}=L^{-1}_2\circ L_1$ is 
an injective continuous map from
the compact set $B_1$
to the Hausdorff space $B_2$.
So $\phi|_{B_1}$ gives a homeomorphism between $B_1$ and $\phi(B_1)$.
Since $O_1$ is an open subset of $\R^2$,
by the invariance of domain, $O'_1:=\phi(O_1)$ is
also an open subset of $O_2$.
Thus, $\phi$ gives a homeomorphism between $O_1$ and $O'_1$.
By the second assertion of Corollary \ref{cor:VU}, 
$L_i|_{O_i}$ ($i=1,2$) are embeddings, and so,
$\phi|_{O_1}$ gives a diffeomorphsm between $O_1$ and $O'_1$.
Since $\epsilon_1(>0)$ can be arbitrarily chosen, 
we can conclude that
$\phi(W_1^*)$ is an open subset of $W_2$, and
$\phi$ is a diffeomorphism 
between $W_1^*$ and $\phi(W_1^*)$.
\end{proof}

Under the preparation above, we prove Theorem A:

\begin{proof}[Proof of Theorem A]
Let $\{q_n\}_{n=1}^\infty$ be a sequence in 
$W_1 \setminus \{p_1\}$ converging to $p_1\in W_1$.
Then $Q_n:=f_1(q_n)$ converges to the origin $\mb 0$ 
by the continuity of $f_1$.
Since $f_1(\overline{W_1})\subset f_2(\overline{W_2})$,
we can take a point $q'_n$ on $\bar W^*_2(\subset \overline{W_2})$ such that
$
f_2(q'_n)=Q_n.
$
Since $\overline{W_2}$ is compact,
$\{q'_n\}$ has an accumulation point 
$q'_\infty\in \overline{W_2}$,
and the continuity of $f_2$ yields
$$
\mb 0=\lim_{n\to \infty}Q_n=\lim_{n\to \infty}f_2(q'_n)=
 f_2(q'_\infty).
$$
By (a2), we have
$q'_\infty=p_2$, which implies that $\phi$ can be uniquely
extended to a continuous map $\phi:W_1\to \phi(W_1)$
satisfying
\begin{equation}\label{eq:p1p2A}
\phi(p_1)=p_2.
\end{equation}
So $f_1=f_2\circ \phi$ holds on $W_1$.
Since $(u_i,v_i)$ ($i=1,2$) are positive coordinate systems,
$\phi$ is orientation preserving.
We set
$$
u:=u_1,\qquad v:=v_1
$$
and
$$ 
x(u,v):=u_2\circ \phi(u,v),\qquad
y(u,v):=v_2\circ \phi(u,v).
$$
Since $(u_i,v_i)$ ($i=1,2$) are both Bruce-West's coordinates,
$f_1=f_2\circ \phi$ implies that
\begin{align}
\label{eq:W1} & x(u,v)=u, \\
\label{eq:W2} & x(u,v)y(u,v)+b_2(y(u,v))=uv+b_1(v)
\end{align}
hold for $(u,v)\in W_1$.
To obtain the conclusion, it is sufficient to
show that $y(u,v)=v$ on a certain open neighborhood $O(\subset W_1)$ of $(0,0)$.
(In fact, suppose that it is true. By setting 
$\psi:=(u_2,v_2)$,
the map $\psi\circ \phi$ coincides with the identity map on $W_1$.
Since $\psi$ is a local coordinate system of $W_2$,
it is a local diffeomorphism, and so,
$\phi(=\psi^{-1})$ is the desired local diffeomorphism
at the origin.

By \eqref{eq:W1} and \eqref{eq:W2}, we have
\begin{equation}
\label{eq:W3} 
uy(u,v)+b_2(y(u,v))=uv+b_1(v).
\end{equation}
We set
$$
b_1(v):=v^3 \beta_1(v), \qquad  b_2(y):=y^3\beta_2(y).
$$
It is well-known that, for any $C^\infty$-function 
$A(x)$ of one variable
defined on an open interval containing $x=0$,
there exists a $C^\infty$-function $B(x,y)$ defined on a 
neighborhood of $(0,0)$ satisfying $A(y)-A(x)=
(y-x)B(x,y)$
on the neighborhood. So there exist $C^\infty$-functions
$\xi(w_1,w_2)$ and $\eta(w_1,w_2)$ defined on 
a neighborhood of $(w_1,w_2)=(0,0)$
such that
\begin{equation}\label{eq:B1B2}
b_2(w_2)-b_2(w_1)=(w_2-w_1)\xi(w_1,w_2), \quad 
\beta_2(w_2)-\beta_2(w_1)=(w_2-w_1)\eta(w_1,w_2). 
\end{equation}
We set
\begin{equation}\label{eq:ZG}
\zeta(u,v):=\xi(v,y(u,v)), \qquad G(u,v):=u+\zeta(u,v),
\end{equation}
which are defined on a rectangular domain
$$
D:=I\times J (\subset W_1)
$$
containing $(0,0)$, where $I$ and $J$ are
open intervals containing the origin $0\in \R$.
By \eqref{eq:W3},
we have
\begin{equation}\label{eq:uv}
b_2(y(u,v))-b_1(v)=
-u(y(u,v)-v).
\end{equation}
On the other hand, we have that
\begin{align}
\label{eq:W4} 
(v-y(u,v))G(u,v)
&=uv-uy(u,v)+(v-y(u,v))\xi(v,y(u,v)) \\ \nonumber
&=\Big(b_2(y(u,v))-b_1(v)\Big)+\Big(b_2(v)-b_2(y(u,v))\Big) 
\nonumber
 \\
&=b_2(v)-b_1(v). \nonumber
\end{align}
By \eqref{eq:uv} and \eqref{eq:W4}, we have
\begin{equation}\label{eq:add}
(y(u,v)-v)\zeta(u,v)
=b_2(y(u,v))-b_2(v),
\end{equation}
which implies 
\begin{align*}
(y(u,v)-v)\zeta(u,v)
&=b_2(y(u,v))-b_2(v) \\
&=y(u,v)^3 \beta_2(y(u,v))-v^3 \beta_2(v) \\
&=y(u,v)^3\Big(\beta_2(y(u,v))- \beta_2(v)\Big)
+\beta_2(v)\Big(y(u,v)^3-v^3\Big) \\
&=y(u,v)^3\Big((y(u,v)-v)\eta(v,y(u,v))\Big)
+\beta_2(v)\Big(y(u,v)^3-v^3\Big)
\end{align*}
on $D$.
So, under the assumption that  $y(u,v)\ne v$, we have
\begin{equation}\label{eq:W5}
\zeta(u,v)=y(u,v)^3\eta \Big(v,y(u,v)\Big)+\beta_2(v)
\Big(y(u,v)^2+y(u,v)v+v^2\Big).
\end{equation}

We now suppose that Theorem A fails.
Then there exists a sequence $\{(u_n,v_n)\}_{n=1}^\infty$
on $D$ 
satisfying
$
y(u_n,v_n)\ne v_n
$
and $\dy{\lim_{n\to \infty}}(u_n,v_n)=(0,0)$.
Since $y(0,0)=0$ (cf. \eqref{eq:p1p2A}),
the equality \eqref{eq:W5} implies
that
$$
\lim_{n\to \infty}\zeta_u(u_n,v_n)=0.
$$
So we may assume that
$G_u(u_n,v_n)\ne 0$ holds for sufficiently large $n$.
We fix such an $n$.
Since $(u_n,v_n)\ne (0,0)$,
by the implicit function theorem,
there exist a neighborhood $O(\subset D)$ 
of $(u_n,v_n)$ 
and a $C^\infty$-function $u:=u(v)$ defined on 
an open interval $J_n$ containing $v_n$
such that
$u(v_n)=u_n$ and 
$$
Z:=\{(u(v),v)\,;\, v\in J_n\} = \{(u,v)
\in O\,;\, G(u,v)=0\}.
$$
Substituting $u=u(v)$ to \eqref{eq:W4},
we have
$$
0=(v-y(u(v),v))G(u(v),v)=b_2(v)-b_1(v),
$$
and so
$b_1=b_2$ holds on $J_n$.
Since $G(u,v)\ne 0$ on $O\setminus Z$,
\eqref{eq:W4} implies that
$
v-y(u,v)=0 
$.
Then, by the continuity of $y$, we have
$$
v-y(u,v)=0 \qquad ((u,v)\in O),
$$
contradicting the fact $y(u_n,v_n)\ne v_n$.
\end{proof}

\begin{proof}[Proof of Theorem B]
Let $f_i$ ($i=1,2$) be two $C^\infty$-maps
defined on a neighborhood $U$ 
of the origin $(0,0)\in \R^2$
each of which has a cross cap singular point at $(0,0)$.
We are thinking that $f_i$ ($i=1,2$) are
map germs of cross cap singularities,
and  assume that $f_2$ is congruent to 
$f_1$ as a map germ (cf. Definition \ref{def:map-c}). Then there 
exist an orientation preserving
isometry $T$ of $\R^3$
and an orientation preserving local diffeomorphism $\phi$ such that
$f_2=T\circ f_1 \circ \phi$.
Without loss of generality, we may assume that
$f_1$ and $f_2$ are written in normal 
forms centered at $(0,0)$, that is,
they are expressed as 
\eqref{eq:west-2b} by setting $u=u_1=u_2$ and $v=v_1=v_2$.

We may regard that each point of $\R^3$
consists of column vectors.
Then $T$ can be identified with an orthogonal matrix
by multiplying them from the left.
Since $T$ maps the tangent line and the principal and
normal planes
of $f$ to those of $g$, it
is written in the form
$$ 
T=\pmt{
\epsilon_1 & 0 & 0 \\
0 & \epsilon_2 & 0 \\
0 & 0 & \epsilon_3 
},
$$
where $\epsilon_j\in \{1,-1\}$ for $j=1,2,3$.
By the condition  $a_{02}>0$, we have $\epsilon_3=1$.
So the possibility of $T$ is
\begin{equation*}
T_0:=\pmt{
1 & 0 & 0\\
0 & 1 & 0\\
0 & 0 & 1
},\quad
T_1:=\pmt{
1 & 0 & 0\\
0 & -1 & 0\\
0 & 0 & 1
},
\quad
T_2:=
\pmt{
-1 & 0 & 0\\
0 & 1 & 0\\
0 & 0 & 1
},
\quad
T_3:=
\pmt{
-1 & 0 & 0\\
0 & -1 & 0\\
0 & 0 & 1
}.
\end{equation*}
Suppose that $T=T_0$ (resp. $T=T_1$). Then
we have
\begin{align}
&T_0\circ f(u,v)=(u,uv+b_1(v),a_1(u,v)) \\
\quad
& \phantom{aaaa}
\Big(\text{resp. }\, T_1\circ f(u,-v)=(u,uv-b_1(-v),a_1(u,-v))\Big).
\end{align}
Comparing this with 
\begin{equation}\label{eq:1f2}
f_2(u,v)=(u,uv+b_2(v),a_2(u,v)),
\end{equation}
we obtain (cf. Theorem A)
\begin{align}\label{eq:0a}
&\phi(u,v)=(u,v),\quad  b_2(v)=b_1(v),\quad a_2(u,v)=a_1(u,v) \\
\label{eq:0b}
& \phantom{aaaa} 
\Big (\text{resp. }
\phi(u,v)=(u,-v),\quad  b_2(v)=-b_1(-v),\quad a_2(u,v)=a_1(u,-v) \Big).
\end{align}

We next suppose that $T=T_2$ ($T=T_3$). Then
we have
\begin{align}
&T_2\circ f(-u,-v)=(u,uv+b_1(-v),a_1(-u,-v)) \\
\quad
& \phantom{aaaa}
\Big(
\text{resp. }\, T_3\circ f(-u,v)=
(u,uv-b_1(v),a_1(-u,v))
\Big).
\end{align}
Again, comparing this with \eqref{eq:1f2},
we obtain 
\begin{align}\label{eq:1a}
&\phi(u,v)=(-u,-v),\quad  b_2(v)=b_1(-v),\quad a_2(u,v)=a_1(-u,-v) \\
\label{eq:1b}
& \phantom{aaaa}
\Big (\text{resp. }
\phi(u,v)=(-u,v),\quad  b_2(v)=-b_1(v),\quad a_2(u,v)=a_1(-u,v) \Big).
\end{align}
If $g$ is positively congruent to $f$, then
$T$ and $\phi$ must be orientation preserving maps.
Since $T_1,T_2$ are orientation reversing,
we have $T=T_0$ or $T_3$.
However, when $T=T_3$, we have seen that $\phi(u,v)=(-u,v)$,
that is, $\phi$ is orientation reversing. So 
only the possibility is the case that
$T$ and $\phi$ are the identity maps.
So we obtain \eqref{eq:0a}, proving the assertion.
\end{proof}

\begin{proof}[Proof of Corollaries C-1, C-2, C-3.]
In the above proof of Theorem B,
we may set $f:=f_1=f_2$.
Then, it can be directly observed that
\begin{itemize}
\item[(t1)] $T_1$ is the reflection with
respect to the principal plane of $f$,
\item[(t2)] $T_2$ is the reflection with
respect to the normal plane,
\item[(t3)] $T_3$ is the $180^\circ$-rotation
with respect to the normal line.
\end{itemize}
Regarding this, we first show that (i) implies (ii).
So we suppose that (i) of Corollaries C-$j$ ($j=1,2,3$) 
hold. Then $T_j(f(V))\subset f(U)$ holds for
a sufficiently small neighborhood $V(\subset U)$
of $p$.
By Fact \ref{thm:west-2} and Theorem A, $T_j\circ f\circ \phi_j=f$
($j=1,2$) holds for a local orientation preserving
diffeomorphism $\phi_j$ on $\R^2$.
Since $T_j$ is an involution,
so is $\phi_j$ are involutions.
Thus, we obtain (ii) of Corollary C-$j$.

If $f$ has the symmetry $T_1$, then
\eqref{eq:0b} implies
$$
-b(-v)=b(v),\quad a(u,-v)=a(u,v).
$$
Similarly, if $f$ has the symmetry $T_2$ (resp. $T_3$), then
\eqref{eq:1a} (resp.\,\, \eqref{eq:1b})
implies
\begin{align*}
& b(v)=b(-v),\quad a(u,v)=a(-u,-v) \\
&\phantom{aaaaaaaaa} 
\Big(\text{resp}.\,\, \, b(v)=0,\quad a(u,v)=a(-u,v)\Big).
\end{align*}
So we get the relations of $a$ and $b$ as in the statements.
Conversely, if $a$ and $b$ satisfy (iii) of 
Corollary C-$j$, then $f$ satisfies (1) obviously.
\end{proof}

\begin{Example}
Consider a family of cross caps 
$
f(u,v):=(u, uv+v^3, cu^2+v^2),
$
where $c$ is a real number.
In this example, the characteristic functions are
given by
$$
b(v)=v^3,\qquad a(u,v)=cu^2+v^2.
$$
Since $T_1\circ f(u,-v)=f(u,v)$,
this example is the case of Corollary C-1.
In fact, $b(v)$ is an odd function
and $v\mapsto a(u,v)$ is an even function. 
\end{Example}

\begin{Example}
Consider a family of cross caps 
$
f(u,v):=(u, uv+v^4, cu^2+v^2),
$
where $c$ is a real number.
In this example, the characteristic functions are
given by
$$
b(v)=v^4,\qquad a(u,v)=cu^2+v^2.
$$
Since $T_2\circ f(-u,-v)=f(u,v)$,
this example is the case of Corollary C-2.
In fact, $b(v)$ 
and $a(u,v)=a(-u,-v)$ are even functions. 
\end{Example}

\begin{Example}
Consider a family of cross caps 
$
f(u,v):=(u, uv, cu^2+v^2),
$
where $c$ is a real number.
The characteristic functions are
given by
$$
b(v)=0,\qquad a(u,v)=cu^2+v^2.
$$
Since $T_3\circ f(-u,v)=f(u,v)$,
this example is the case of Corollary C-3.
In fact, $b(v)$ vanishes identically
and $a(u,v)=a(-u,v)$ holds. 
Since $a(u,v)$ is an even function with respect to $u$
and $v$, this example has the property that
$$
T_1\circ f(u,-v)=f(u,v),\qquad T_2\circ f(-u,-v)=f(u,v),
$$
that is, the image of $f$ is invariant under
the three isometries $T_1,\,T_2$ and $T_3$.
\end{Example}

\begin{proof}[Proof of Proposition D]
We fix an orientable Riemannian $3$-manifold $(M^3,g)$
with a fixed point $o\in M^3$.
Let $\{e_1,e_2,e_3\}$ be an 
oriented orthonormal basis of
the tangent space $T_oM^3$ of $M^3$ at $o$,
and let $(x_1,x_2,x_3)$ be 
a geodesic normal coordinate system of
$M^3$ at $o$ induced by the exponential map at $o$.
We let $f:U\to M^3$ be a $C^\infty$-map which has
a cross cap singular point $p\in U$
satisfying $f(p)=o$.
Then by a suitable choice of 
$\{e_1,e_2,e_3\}$ and by a choice of
an oriented local coorodinate system $(u,v)$, 
we may assume that
$f(u,v)=(x_1(u,v),x_2(u,v),x_3(u,v))$
is written as
\begin{equation}\label{eq:forD}
x_1(u,v)=u,\quad
x_2(u,v)=uv +b(v),\quad
x_3(u,v)=a(u,v).
\end{equation}
We then consider the case that
$(M^3,g)$ is $M^3(c)$ with constant curvature $c(\in \R)$.
Regarding $T_oM^3(c)$ as a column vector space,
each isometry of $M^3(c)$ fixing $o$ corresponds
to a left-multiplication of a $3\times 3$ orthogonal matrix.
Moreover,
the left-multiplications of the orthogonal matrices $T_i$ $(i=1,2,3)$
on the coordinate system $(x_1,x_2,x_3)$
give isometric motions of $M^3(c)$ fixing $f(p)$.
Furthermore, the matrices $T_1,\,\,T_2$ and $T_3$
satisfy the conditions (t1), (t2) and (t3),
respectively.
Thus the assertions of Proposition D can be
proved by modifying the proofs of 
Corollaries C-$1$, C-$2$, C-$3$.
\end{proof}

\begin{Remark}
Even when $(M^3,g)$ is an arbitrarily given Riemannian 
$3$-manifold,
the functions $a(u,v)$ and $b(v)$ 
given in \eqref{eq:forD}
can be considered
as geometric invariants for cross caps.
(In fact,
the choice of $\{e_1,e_2,e_3\}$
corresponds to the coordinate changes of normal coordinates
$(x_1,x_2,x_3)$ via a left-multiplication of the orthogonal matrix
with determinant $1$.
Thus, $a(u,v)$ and $b(v)$ are determined independently 
of  a choice of the local coordinates 
$(u,v)$ and of
a choice of $\{e_1,e_2,e_3\}$ 
giving the expression \eqref{eq:forD} for $f$
as a consequence of Theorem A.)
\end{Remark}

\begin{acknowledgments}
The authors thank  Toshizumi Fukui for helpful comments.
\end{acknowledgments}


\begin{thebibliography}{00}

\bibitem{BT}
J. W. Bruce and F. Tari,
{\it On binary differential equations},
Nonlinearity {\bf 8} (1995), 255--271. 

\bibitem{BW}
 J. W. Bruce and J. M. West,
{\it Functions on a crosscap},
Math. Proc. Cambridge Philos. Soc. {\bf 123} (1998), 19--39.

\bibitem{I6}
F. S. Dias and F. Tari,
{\it On the geometry of the cross-cap in the Minkoswki $3$-space 
and binary differential equations},
Tohoku Math. J. (2) {\bf 68} (2016), 293--328. 


\bibitem{FH}
T.~Fukui, M.~Hasegawa, 
{\it Fronts of Whitney umbrella -- differential geometric approach via blowing
up}, J. Singul.  {\bf 4} (2012), 35--67. 

\bibitem{I1}
T. Fukui and J. J. Nuno Ballesteros, 
{\it Isolated roundings and flattennings of submanifolds in Euclidian 
paces}, Tohoku Math. J. {\bf 57} (2005), 469--503. 

\bibitem{I2}
R. Garcia, C. Gutierrez and J. Sotomayor, 
{\it Lines of Principal curvature around umbilics and Whitney umbrellas}, 
Tohoku Math. J. {\bf 52} (2000), 163--172. 

\bibitem{GG}
 M. Golubitsky and  V.\ Guillemin,
 {\rm Stable mappings and their singularities},
 Graduate Texts in Mathematics, {\bf 14}, Springer-Verlag, 1973. 

\bibitem{I3}
C. Gutierrez and J. Sotomayor, 
{\it Lines of principal curvature for mappings with Whitney umbrella 
ingularities}, 
Tohoku Math. J. {\bf 38} (1986), 551--559. 


\bibitem{HHNSUY-2015}
  M. Hasegawa, A. Honda, K. Naokawa,
  K. Saji, M. Umehara and K. Yamada,
  \emph{Intrinsic properties of surfaces with singularities},
  Int.\ J. Math. {\bf 26} (2015), 1540008.

\bibitem{HHNUY-2014}
M. Hasegawa, A. Honda, K. Naowaka, M. Umehara and K. Yamada,
{\it 
Intrinsic invariants of cross caps},
Selecta Math. New Ser. {\bf 20} (2014), 769--785.

\bibitem{HNSUY}
A. Honda, K. Naokawa,
  K. Saji, M. Umehara and K. Yamada,
  \emph{
A generalization of Zakalyukin's lemma, and 
symmetries of surface singularities},
to apper in J. Singul. (arXiv:2104.03505).

\bibitem{HNUY}
A. Honda, K. Naokawa, M. Umehara, K. Yamada,
{\it 
Isometric realization of cross caps as formal 
power series and its applications},
Hokkaido Math. J. {\bf 48} (2019), 1--44.

\bibitem{I4}
J. M. Oliver,
{\it On pairs of foliations of a parabolic cross-cap}, 
Qual. Theory  Dyn. Syst. {\bf 10} (2011), 139--166. 


\bibitem{I5}
F. Tari, 
{Pairs of geometric foliations on a cross-cap}, 
Tohoku Math. J. {\bf 59} (2007), 233--258. 

\bibitem{west}
J. West,
{\it The differential geometry of the cross-cap},
Ph. D. thesis, University of Liverpool (1995)

\end{thebibliography}
\end{document}